\documentclass[10pt]{amsart}
\usepackage{amssymb,MnSymbol,times}
\usepackage{amsthm,amsmath,cite,color}
\usepackage{tikz}
\usetikzlibrary{arrows,shapes,backgrounds,decorations}

\title[On the structure of symmetric $n$-ary bands]{On the structure of symmetric $n$-ary bands}

\author{Jimmy Devillet}
\address{University of Luxembourg, Department of Mathematics, Maison du Nombre, 6, avenue de la Fonte, L-4364 Esch-sur-Alzette, Luxembourg}
\email{jimmy.devillet[at]uni.lu}

\author{Pierre Mathonet}
\address{University of Li\`ege, Department of Mathematics, All\'ee de la D\'ecouverte, 12 - B37, B-4000 Li\`ege, Belgium}
\email{p.mathonet[at]uliege.be}

\date{April 26, 2020}

\begin{document}

\theoremstyle{plain}
\newtheorem{theorem}{Theorem}[section]
\newtheorem{lemma}[theorem]{Lemma}
\newtheorem{proposition}[theorem]{Proposition}
\newtheorem{corollary}[theorem]{Corollary}
\newtheorem{fact}[theorem]{Fact}
\newtheorem{conjecture}[theorem]{Conjecture}
\newtheorem*{main}{Main Theorem}

\theoremstyle{definition}
\newtheorem{definition}[theorem]{Definition}
\newtheorem{example}[theorem]{Example}
\newtheorem{algorithm}{Algorithm}

\theoremstyle{remark}
\newtheorem{remark}{Remark}
\newtheorem{claim}{Claim}

\newcommand{\N}{\mathbb{N}}
\newcommand{\Z}{\mathbb{Z}}
\newcommand{\R}{\mathbb{R}}
\newcommand{\ran}{\mathrm{ran}}
\newcommand{\Cdot}{\boldsymbol{\cdot}}
\newcommand{\cw}{\curlywedge}
\renewcommand{\l}{\ell}

\begin{abstract}
We study the class of symmetric $n$-ary bands. These are $n$-ary semigroups $(X,F)$ such that $F$ is invariant under the action of permutations and idempotent, i.e., satisfies $F(x,\ldots,x)=x$ for all $x\in X$. 
We first provide a structure theorem for these symmetric $n$-ary bands that extends the classical (strong) semilattice decomposition of certain classes of bands. We introduce the concept of strong $n$-ary semilattice of $n$-ary semigroups and we show that the symmetric $n$-ary bands are exactly the strong $n$-ary semilattices of $n$-ary extensions of Abelian groups whose exponents divide $n-1$. Finally, we use the structure theorem to obtain necessary and sufficient conditions for a symmetric $n$-ary band to be reducible to a semigroup.
\end{abstract}

\keywords{Semigroup, polyadic semigroup, semilattice decomposition, reducibility, idempotency}

\subjclass[2010]{Primary 20M10, 20M14, 20N15; Secondary 08A30, 20K25}

\maketitle

\section{Introduction}\label{Sec1}
A band is a semigroup $(X,G)$ where $X$ is a nonempty set and the binary associative operation $G \colon X^2\to X$ is idempotent, i.e., satisfies $G(x,x)=x$ for every $x\in X$. Bands and more generally semigroups have been extensively studied from the second half of the 20th century, starting with the works of Clifford \cite{Cli54} and McLean \cite{Mac54}, among others. We refer the reader to \cite{Cli61,Pet73,Pet77,How95} or more recently \cite{Gri01} for more information. One of the first results in the structural theory of semigroups is the existence of semilattice decompositions \cite{Gri01}.
This concept can be further refined to define strong semilattice decompositions (see Section \ref{Sec2.2} for details), which proved to be very useful in the analysis of particular classes of semigroups. In this paper we generalize strong semilattice decompositions to structures with higher arities. 

The concept of $n$-ary semigroup (where $n\geqslant 2$ is an integer) was introduced by D\"{o}rnte \cite{Dor28} and further studied by Post \cite{Pos40} in the framework of $n$-ary groups and their reductions.
It is defined as follows. An $n$-ary operation $F\colon X^n\to X$ is \emph{associative} if 
\begin{multline}\label{assoc}
F(x_1,\ldots,x_{i-1},F(x_i,\ldots,x_{i+n-1}),x_{i+n},\ldots,x_{2n-1})\\
=~ F(x_1,\ldots,x_i,F(x_{i+1},\ldots,x_{i+n}),x_{i+n+1},\ldots,x_{2n-1}),
\end{multline}
for all $x_1,\ldots,x_{2n-1}\in X$ and all $1\leq i\leq n-1$. If $F$ is an $n$-ary associative operation on $X$, then $(X,F)$ is an $n$-ary semigroup.
Symmetry and idempotency can also be extended to $n$-ary operations as follows: $F$ is \emph{idempotent} if $F(x,\ldots,x)=x$ for every $x\in X$ and $F$ is \emph{symmetric} (or \emph{commutative}) if $F$ is invariant under the action of permutations.

It is easy to provide examples of such $n$-ary semigroups by extending binary semigroups: if $G\colon X^2\to X$ is an associative operation, then we can define a sequence of operations inductively by setting $G^1=G$, and
$$
G^m(x_1,\ldots,x_{m+1}) ~=~ G^{m-1}(x_1,\ldots,x_{m-1},G(x_m,x_{m+1})), \qquad m\geq 2.
$$
Setting $F=G^{n-1}$, the pair $(X,F)$ is then an $n$-ary semigroup. It is the \emph{$n$-ary extension} of $(X,G)$\footnote{We also say that $F$ is the $n$-ary extension of $G$ or that $F$ is reducible to $G$ or even that $G$ is a binary reduction of $F$.}. There are also $n$-ary semigroups that are not reducible to binary semigroups. A simple example is given by the ternary operation $F$ defined on $\R^3$ by $F(x_1,x_2,x_3)=x_1-x_2+x_3$. The problem of reducibility was considered in a series of recent papers. In \cite{DudMuk06} necessary and sufficient conditions for reducibility were given in terms of neutral elements (see Subsection \ref{Sec2.1}). In \cite{KiSom18,LehSta19} this problem was considered in the context of sets $X$ endowed with additional structure. More recently, in \cite{Ack20,CouDev,DevKisMar19} the class of quasitrivial $n$-ary semigroups was investigated. These are $n$-ary semigroups $(X,F)$ such that $F(x_1,\ldots,x_n)\in\{x_1,\ldots,x_n\}$ for all $x_1,\ldots,x_n\in X$. For $n=2$, the quasitrivial semigroups were described by L\"{a}nger \cite{Lan80}. It was shown \cite{CouDev} that all quasitrivial $n$-ary semigroups are reducible. These $n$-ary semigroups are clearly idempotent. In \cite{CouDevMarMat19}, the authors relaxed the quasitriviality condition by considering operations whose restrictions on certain subsets of the domain  are quasitrivial. These operations are also reducible.

In the present work, we consider the class of symmetric idempotent $n$-ary semigroups. By analogy with the terminology used for binary structures we call them symmetric (or commutative) $n$-ary bands. The process of extension described above enables us to present two important examples of symmetric $n$-ary bands: if $(X,\cw)$ is a semilattice, then we can extend the semilattice operation to $X^n$ using associativity, and define an $n$-ary operation $F\colon X^n\to X$ by $F(x_1,\ldots,x_n)=x_1\cw\cdots\cw x_n$. Since the extension procedure preserves associativity, symmetry and idempotency, this operation defines a symmetric $n$-ary band on $X$.
Another class of symmetric $n$-ary bands is given by considering $n$-ary extensions of groups: if $(X,\ast)$ is an Abelian group whose exponent divides $n-1$, then we can extend the group operation $\ast$ to an $n$-ary operation by setting $F(x_1,\ldots,x_n)=x_1\ast\cdots\ast x_n$. The hypotheses on the group operation ensure that this operation defines a symmetric $n$-ary band. Both classes of examples are made up of $n$-ary operations that are reducible to binary symmetric semigroup operations. They will play a central role in our constructions. However, this raises the natural question of existence of symmetric $n$-ary bands that are not $n$-ary extensions of commutative binary semigroups. The answer to this question is positive, as shown by the following examples.
\begin{example}\label{ex:1.1}
\begin{enumerate}\item[(a)]
 We consider the set $X=\{1,2,3,4\}$ and we define the symmetric ternary operation $F_1\colon X^3\to X$ by its level sets given (up to permutations) by $F_1^{-1}(\{1\})=\{(1,1,1)\}$, $F_1^{-1}(\{2\})=\{(2,2,2)\}$, $F_1^{-1}(\{3\})=\{(1,1,2),(1,1,3),$\\$(1,2,4),(1,3,4),(2,2,3),(2,3,3),(2,4,4),(3,3,3),(3,4,4)\}$. Then $F_1^{-1}(\{4\})$ is made up of all the remaining elements of $X^3$. This operation defines a symmetric ternary band and is not reducible to any binary operation.
\item[(b)]
 We still consider the set $X=\{1,2,3,4\}$ and we define the symmetric ternary operation $F_2\colon X^3\to X$ again by its level sets given (up to permutations) for the first three elements by $F_2^{-1}(\{1\})=\{(1,1,1)\}$, $F_2^{-1}(\{2\})=\{(2,2,2)\}$, $F_2^{-1}(\{3\})=$\\$\{(1,1,3),(1,2,3),(1,3,4),(2,2,3),(2,3,4),(3,3,3),(3,4,4)\}$. Then $F_1^{-1}(\{4\})$ is again made up of all the remaining elements. This operation defines a symmetric ternary band. It turns out that it is reducible to a binary operation on $X$.\end{enumerate}
\end{example}
These examples can be checked by hands, by tedious computations. We will develop tools (namely the strong semilattice decompositions of symmetric $n$-ary bands) that will enable us to check these properties and to build such examples very easily.

 More precisely, we show in Section \ref{Sec3} that we can associate with any symmetric $n$-ary band $(X,F)$ a (binary) band $(X,B)$, that is in general not commutative, but is right normal. We study the properties of this band and provide necessary and sufficient conditions for $(X,F)$ to be the $n$-ary extension of a semilattice, or of an Abelian group. In Section \ref{Sec4} we show that the strong semilattice decomposition of this band (see Subsection \ref{Sec2.2}) induces a decomposition of the $n$-ary band $(X,F)$. This leads us to introduce the concept of strong $n$-ary semilattice decomposition of a (symmetric) $n$-ary band. The restriction of $F$ to each subset of this decomposition defines an $n$-ary semigroup. We show that this semigroup is the $n$-ary extension of a commutative group whose exponent divides $n-1$. Finally, we show that the converse construction can be carried out: if $(X,B)$ is a right normal band and if each subset of its greatest semilattice decomposition is endowed with a commutative group structure whose exponent divides $n-1$, that is compatible with $(X,B)$, then we can build a symmetric $n$-ary band $(X,F)$ whose associated binary band is $(X,B)$. These two constructions provide a structure theorem for symmetric $n$-ary bands (see Theorem \ref{thm:mainnband}). Finally, in Section \ref{Sec5} we give necessary and sufficient conditions for a symmetric $n$-ary band to be reducible to a semigroup using its $n$-ary semilattice decomposition.

\section{Notation and standard constructions}\label{Sec2}
In this section we fix the notation and recall results that will be useful for our developments. These results come from the general theory of $n$-ary semigroups on the one hand and from the classical theory of semigroups on the other hand. 
\subsection{General theory of $n$-ary semigroups}\label{Sec2.1}
Throughout this work we consider a nonempty set $X$ and an integer $n\geqslant 2$. Then $(X,F)$ is an \emph{$n$-ary groupoid} if $F\colon X^n\to X$ is an $n$-ary operation. 
If $F$ is associative (i.e., satisfies \eqref{assoc}), then $(X,F)$ is an $n$-ary semigroup.

As usual, a homomorphism of $n$-ary groupoids from $(X,F)$ to $(Y,F')$ is a map $\varphi\colon X \to Y$ such that
\[
\varphi(F(x_1,\ldots,x_n)) = F'(\varphi(x_1),\ldots,\varphi(x_n)),\qquad x_1,\ldots,x_n\in X.
\] 
If $(X,F)=(Y,F')$ we also say that $\varphi$ is an endomorphism on $(X,F)$.
Two $n$-ary groupoids $(X,F)$ and $(Y,F')$ are isomorphic if there exists a bijective homomorphism from $(X,F)$ to $(Y,F')$. 

The extension process described in Section \ref{Sec1} for semigroups can be adapted for $n$-ary semigroups. Indeed, for any $n$-ary associative operation $F\colon X^n\to X$, a sequence $(F^q)_{q\geq 1}$ of $(qn-q+1)$-ary operations can be defined inductively by the rules $F^1 = F$ and
\begin{equation}\label{eq:extension}
F^{q}(x_1,\ldots,x_{qn-q+1}) ~=~ F^{q-1}(x_1,\ldots,x_{(q-1)n-q+1},F(x_{(q-1)n-q+2},\ldots,x_{qn-q+1})),
\end{equation}
for any integer $q\geq 2$ and any $x_1,\ldots,x_{qn-q+1} \in X$. It is straightforward to see that $F^q$ is associative. Moreover, it is idempotent whenever $F$ is idempotent. Also, it was shown in \cite[Lemma 2.10]{LehPil14} that $F^q$ is symmetric whenever $F$ is symmetric.

Concerning binary reductions, we will use the criterion provided in \cite{DudMuk06}, using the concept of neutral element.
Recall that $e\in X$ is said to be \emph{neutral} for $F\colon X^n\to X$ if
$$
F((k-1)\Cdot e,x,(n-k)\Cdot e) ~=~ x,\qquad x\in X,~k\in\{1,\ldots,n\},
$$
where, for any $k\in\{0,\ldots,n\}$ and any $x\in X$, the notation $k\Cdot x$ stands for the $k$-tuple $x,\ldots,x$ (for instance $F(3\Cdot x,0\Cdot y,2\Cdot z)=F(x,x,x,z,z)$).

It was first proved \cite[Lemma 1]{DudMuk06} that if an associative operation $F\colon X^n \to X$ has a neutral element $e$, then it is reducible to the associative operation $G_{e}\colon X^2\to X$ defined by
\begin{equation}\label{eq:dud}
G_{e}(x,y) ~=~ F(x,(n-2)\Cdot e,y),\qquad x,y \in X.
\end{equation}
Moreover, Theorem 1 of \cite{DudMuk06} states that an associative $n$-ary operation is reducible to an associative binary operation if and only if one can adjoin to $X$ a neutral element $e$ for $F$; that is, there is an $n$-ary associative operation $F^*$ defined on $\left(X\cup\{e\}\right)^n$ that extends $F$ and for which $e$ is a neutral element.

Throughout this work, we denote the set of neutral elements for $F$ by $E_F$.

Finally, recall that an equivalence relation $\sim$ on $X$ is a \emph{congruence} for $F\colon X^n\to X$ (or on $(X,F)$) if it is compatible with $F$, i.e., if $F(x_{1},\ldots,x_{n}) \sim F(y_{1},\ldots,y_{n})$ for any $x_1,\ldots,x_n,y_1,\ldots,y_n\in X$ such that $x_i\sim y_i$ for all $i\in \{1,\ldots,n\}$. We denote by $[x]_{\sim}$ or simply by $[x]$ the class of $x$ with respect to $\sim$ and  by $\tilde{F}$ the map induced by $F$ on $X/\sim$ defined by
\[
\tilde{F}([x_1]_{\sim},\ldots,[x_n]_{\sim})=[F(x_1,\ldots,x_n)]_{\sim},\quad\forall x_1,\ldots,x_n\in X.
\]

\subsection{Semilattice decompositions of semigroups}\label{Sec2.2}
In this section, we recall the most important semigroup constructions that we will use in the following sections. Details may be found in \cite{Cli61,How95,Pet73,Pet77,Qui17}.
In what follows, we denote a binary operation on a set $X$ by $G$ or $B$ or simply by $\ast$. If we want to insist on the fact that this operation defines a semilattice, we also denote it by $\cw$.

A congruence on a groupoid $(X,G)$ is a \emph{semilattice congruence} if $(X/\sim,\tilde{G})$ is a semilattice. The partition of $X$ induced by a semilattice congruence is called a \emph{semilattice decomposition of $X$}. It is well known \cite{Put,Tam} that every semigroup admits a smallest semilattice congruence. The associated decomposition is the greatest semilattice decomposition. For a band $(X,G)$ the smallest semilattice congruence $\sim$ is given by (see \cite{Pet77})
\begin{equation}\label{leastcong}
x\sim y \quad \Leftrightarrow \quad G(G(x,y),x)=x \quad \mbox{and} \quad G(G(y,x),y)=y, \qquad x,y\in X.
\end{equation}
Thus a semilattice congruence on $(X,G)$ induces a decomposition of the set $X$ as a disjoint union of the family of subsets $\{X_{\alpha}\colon\alpha\in Y\}$ where $(Y,\cw)$ is a semilattice. The operation $G$ induces on each $X_\alpha$ a semigroup operation $G_\alpha$. Moreover, we have 
\begin{equation}\label{eq:semisem}
G(X_{\alpha}\times X_{\beta}) \subseteq X_{\alpha \curlywedge \beta}, \qquad \alpha,\beta\in Y.
\end{equation}
Such a structure is called a \emph{semilattice $(Y,\curlywedge)$ of semigroups $(X_{\alpha},G_{\alpha})$} and we write $(X,G) = ((Y,\curlywedge);(X_{\alpha},G_{\alpha}))$. Obviously the fact that a groupoid is a semilattice of semigroups does not ensure that it is itself a semigroup. Indeed, by \eqref{eq:semisem} we only know that $G(x,y)\in X_{\alpha \curlywedge \beta}$ for any $(x,y)\in X_{\alpha}\times X_{\beta}$ but this relation does not constrain $G(x,y)$ sufficiently to ensure associativity. In order to get additional information we need to introduce the concept of strong semilattice of semigroups (see \cite{Pet77}).
\begin{definition}\label{def:strong}
Let $(X,G) = ((Y,\curlywedge);(X_{\alpha},G_{\alpha}))$ be a semilattice of semigroups. Suppose that for any $\alpha,\beta\in Y$ such that $\alpha\geqslant \beta$ (i.e., $\alpha\cw\beta=\beta$) there is a homomorphism $\varphi_{\alpha,\beta}\colon X_{\alpha} \to X_{\beta}$ such that the following conditions hold.
\begin{enumerate}
\item[(a)] The map $\varphi_{\alpha,\alpha}$ is the identity on $X_{\alpha}$.
\item[(b)] For any $\alpha,\beta,\gamma\in Y$ such that $\alpha \geqslant \beta \geqslant \gamma$ we have $\varphi_{\beta,\gamma}\circ \varphi_{\alpha,\beta} = \varphi_{\alpha,\gamma}$.
\item[(c)] For any $x\in X_{\alpha}$ and any $y\in X_{\beta}$ we have $G(x,y) = G_{\alpha\curlywedge \beta}(\varphi_{\alpha,\alpha \curlywedge \beta}(x),\varphi_{\beta,\alpha\curlywedge \beta}(y))$.
\end{enumerate}
Then $(X,G)$ is said to be a \emph{strong semilattice $(Y,\curlywedge)$ of semigroups $(X_{\alpha},G_{\alpha})$}. In this case we write $(X,G) = ((Y,\curlywedge);(X_{\alpha},G_{\alpha});\varphi_{\alpha,\beta})$ (or simply $[Y,X_\alpha,\varphi_{\alpha,\beta}]$) and $(X,G)$ is called a strong semilattice of semigroups.
\end{definition}
As explained in \cite[p.16]{Pet77}, this definition can be seen in two different ways: starting with a semigroup $(X,G)$, we can decompose it using a semilattice congruence, and the structure $(X,G)$ induces the homomorphisms $\varphi_{\alpha,\beta}$. But starting the other way, given a family of semigroups indexed by a semilattice $(Y,\cw)$ and homomorphisms satisfying conditions (a) and (b) of Definition \ref{def:strong}, if we define a multiplication by condition (c), then we have a semigroup whose multiplication induces the given data: any strong semilattice of semigroups is a semigroup (see \cite[p.89]{How95} for a detailed proof).

In the next sections we will deal with bands having additional properties. Recall that a band $(X,G)$ is \emph{right normal} if $G(G(x,y),z) = G(G(y,x),z)$ for any $x,y,z\in X$. If $(X,G)$ is a right normal band, then its least semilattice congruence $\sim$ can be characterized more easily: we have
\begin{equation}\label{eq:sigma}
x\sim y \quad \Leftrightarrow \quad G(y,x)=x \quad\mbox{and}\quad G(x,y)=y, \qquad x,y\in X.
\end{equation}
We will also make use of \emph{right zero} semigroups. A right zero semigroup is a semigroup $(X,G)$ where $G(x,y)=y$ for all $x,y\in X$ (see \cite[p.29]{Pet77}). They appear naturally in the strong semilattice decomposition of right normal bands as shown by the following proposition. 
\begin{proposition}[{see \cite{How95}}]\label{prop:rightn}
A band $(X,G)$ is right normal if and only if it is a strong semilattice of right zero semigroups.
\end{proposition}
Moreover, the decomposition of a right normal band $(X,G)$ can be given explicitly: the semilattice congruence $\sim$ is defined by \eqref{eq:sigma}. The semilattice $(Y,\cw)$ is then $(X/\sim,\tilde{G})$. We have $[x]_\sim\geqslant[y]_\sim$ if and only if $[G(x,y)]_\sim=[y]_\sim$, which is also equivalent to $G(x,y)=y$ in this particular situation. The subsets $X_\alpha$ in the definition are the equivalences classes (seen as subsets of $X$). The homomorphisms\footnote{Recall that homomorphisms between right zero semigroups are just mappings.} are then defined by $\varphi_{[x]_\sim,[y]_\sim}(z)=G(y,z)$ for every $z\in [x]_\sim$. 

\section{The associated binary band}\label{Sec3}
Throughout this section, we consider a symmetric $n$-ary band $(X,F)$. We associate with it a classical (binary) band and study its most important properties. Let us start with a definition.
\begin{definition}\label{def:binary}
The binary operation $B_F\colon X^2\to X$ associated with $F$ is defined by 
\[B_F(x,y)=F((n-1)\Cdot x,y),\quad x,y\in X.\]
For every $x\in X$, we also define the operation $\l_x^F\colon X\to X$ by 
\[\l_x^F(y)=B_F(x,y), \quad y\in X.\]
\end{definition}
When there is no risk of confusion, we also denote these operations by $B$ and $\l_x$, respectively. We now study elementary properties of these maps.
\begin{proposition}\label{prop:end}
For every $x,x_1,\ldots,x_n\in X$ we have $\l_x^2=\l_x$ and 
\begin{equation}\label{eq1}\l_x(F(x_1,\ldots,x_n))=F(x_1,\ldots,\l_x(x_i),\ldots,x_n),\qquad i\in\{1,\ldots,n\}.\end{equation}
Moreover, the map $\l_x$ is an endomorphism on $(X,F)$.
\end{proposition}
\begin{proof}
For every $x,y\in X$ we have
\[\l_x^2(y)=F((n-1)\Cdot x,\l_x(y))=F((n-1)\Cdot x,F((n-1)\Cdot x,y)).\]
By the associativity and idempotency of $F$, this expression is equal to \[F(F(n\Cdot x),(n-2)\Cdot x, y)=F((n-1)\Cdot x,y)=\l_x(y).\]
Let us show that \eqref{eq1} holds for $i=1$. The other cases are obtained from the symmetry of $F$. Using the definition of $\l_x$ and the associativity of $F$ we have 
\begin{eqnarray*}
\l_x(F(x_1,\ldots,x_n)) &=& F((n-1)\Cdot x,F(x_1,\ldots,x_n)) \\
&=& F(F((n-1)\Cdot x,x_1),x_2,\ldots,x_n)= F(\l_x(x_1),x_2,\ldots,x_n). 
\end{eqnarray*}
Finally, using \eqref{eq1} and $\l_x^2=\l_x$, we have
\[\l_x(F(x_1,\ldots,x_n))=\l_x^n(F(x_1,\ldots,x_n))=F(\l_x(x_1),\ldots,\l_x(x_n)),\]
which concludes the proof.
\end{proof}
From the idempotency of $F$ we derive $\l_x(x)=x$, for every $x\in X$. We then obtain the following corollary.
\begin{corollary}\label{cor:proj}
For every $x_1,\ldots,x_n\in X$ we have
\[F(x_1,\ldots,x_n)=F(\l_{F(x_1,\ldots,x_n)}(x_1),\ldots,\l_{F(x_1,\ldots,x_n)}(x_n)).\]
\end{corollary}
We now show that the groupoid $(X,B)$ associated with $(X,F)$ is a right normal band.
\begin{proposition}\label{prop:impo}
We have 
\[\l_x\circ \l_y=\l_y\circ \l_x=\l_{\l_x(y)}=\l_{\l_y(x)}, \quad x,y\in X.\]
In other words, the pair $(X,B)$ is a right normal band. 
\end{proposition}
\begin{proof}
For every $x,y,z\in X$ we use \eqref{eq1} to obtain
\[\l_x(\l_y(z))=\l_x(F((n-1)\Cdot y,z))=F((n-1)\Cdot y,\l_x(z))=\l_y(\l_x(z)).\]
The same relation (applied $n-1$ times) yields
\[\l_x(\l_y(z))=\l_x(F((n-1)\Cdot y,z))=\l_x^{n-1}(F((n-1)\Cdot y,z))=F((n-1)\Cdot \l_x(y),z)=\l_{\l_x(y)}(z).\]
The last relation is obtained by exchanging the roles of $x$ and $y$.

Expressing these conditions for $B$, we have 
\[B(B(x,y),z)=\l_{\l_x(y)}(z)=\l_x(\l_y(z))=B(x,B(y,z)),\quad x,y,z\in X,\] so $B$ is associative. Moreover, $B(x,x)=F(n\Cdot x)=x$ for any $x\in X$, so $(X,B)$ is a band. Finally, 
\[B(B(x,y),z)=\l_x(\l_y(z))=\l_y(\l_x(z))=B(B(y,x),z),\quad x,y,z\in X,\] and $B$ is a right normal band.
\end{proof}
We derive the following direct corollary. 
\begin{corollary}\label{cor:semi}
The pair $(\{\l_x\colon x\in X\},\circ)$ is a semilattice.
\end{corollary}

\begin{example}\label{ex:2.6}The binary bands associated with the ternary bands $(X,F_1)$ and $(X,F_2)$ defined in Example \ref{ex:1.1}   are given by the following tables:
\begin{center}
\begin{tabular}{l@{\hspace{2cm}}r}
$\begin{array}{c|cccc}
 B_{F_1}&1&2&3&4\\\hline
 1&1&3&3&4\\
 2&4&2&3&4\\
 3&4&3&3&4\\
 4&4&3&3&4
\end{array}$
 &
$ \begin{array}{c|cccc}
 B_{F_2}&1&2&3&4\\\hline
 1&1&4&3&4\\
 2&4&2&3&4\\
 3&4&4&3&4\\
 4&4&4&3&4
\end{array}$

\end{tabular}
\end{center} 
\end{example}
The semilattice defined in Corollary \ref{cor:semi} can be extended to define a symmetric $n$-ary band. The following result establishes a tight relation between $(X,F)$ and this $n$-ary band.
\begin{proposition}\label{prop:impo3}
For every $x_1,\ldots,x_n\in X$ we have 
\[\l_{F(x_1,\ldots,x_n)}=\l_{x_1}\circ \cdots \circ \l_{x_n}.\] 
\end{proposition}
\begin{proof}
We use the $n$-ary extension $F^n$ of $F$ defined by \eqref{eq:extension} to compute the right hand side. By the symmetry of $F^n$ we have for any $t\in X$,
\begin{multline*}
(\l_{x_1}\circ \cdots \circ \l_{x_n})(t) ~=~ F^n((n-1)\Cdot x_1,\ldots,(n-1)\Cdot x_n,t) \\
=F((n-1)\Cdot F(x_1,\ldots,x_n),t) ~=~ \l_{F(x_1,\ldots,x_n)}(t),
\end{multline*}
which completes the proof. 
\end{proof}
The next result characterizes the reducibility of a symmetric $n$-ary band to a symmetric (binary) band, i.e., a semilattice. In order to state them, we consider the map\footnote{Here $T_X$ denotes the full transformation monoid of $X$, i.e., the semigroup of all maps from $X$ to $X$ endowed with the composition of maps.} $\l\colon X\to T_X$ defined by $\l(x)=\l_x$, for every $x\in X$. 
\begin{proposition}\label{prop:propinj1}
Let $(X,F)$ be a symmetric $n$-ary band. The following assertions are equivalent.
\begin{enumerate}
\item[(i)] The map $\l$ is injective.
\item[(ii)] The $n$-ary band $(X,F)$ is (isomorphic to) the $n$-ary extension of a semilattice.
\item[(iii)] The band $(X,B)$ is commutative.
\item[(iv)] The $n$-ary band $(X,F)$ is the $n$-ary extension of $(X,B)$.
\end{enumerate}
\end{proposition}
\begin{proof}
Let us first show that (i) implies (ii). First remark that the property of being the $n$-ary extension of a semilattice is preserved by isomorphisms. Now, if $\l$ is injective, then it induces a bijection from $X$ to $\{\l_x\colon x\in X\}$. Also, Proposition \ref{prop:impo3} shows that $\l$ is an isomorphism from $(X,F)$ to the $n$-ary extension of the semigroup $(\{\l_x\colon x\in X\},\circ)$, which is a semilattice by Corollary \ref{cor:semi}. 
    Let us show that (ii) implies conditions (iii) and (iv). Assume that $(X,F)$ is the $n$-ary extension of a semilattice $(X,\cw)$. Then for any $x,y\in X$ we have
\[
B(x,y) = F((n-1)\Cdot x,y) = x\cw\cdots\cw x\cw y=x\cw y
\]
which shows that $B$ is commutative and that $(X,F)$ is the $n$-ary extension of $(X,B)$. Let us show that (iii) implies (i). So, assume that $B$ is commutative and that $\l_x=\l_y$ for some $x,y\in X$. Then we have $B(x,z)=B(y,z)$ for every $z\in X$, and thus
\[x=B(x,x)=B(y,x)=B(x,y)=B(y,y)=y,\]
which shows that $\l$ is injective. Finally, let us show that (iv) implies (iii).  If $F$ is the $n$-ary extension of $B$, since $B$ is idempotent we have for any $x,y\in X$,
\[F(y,(n-1)\Cdot x)=B(y,x)\quad\mbox{and}\quad F((n-1)\Cdot x,y)=B(x,y).\]
Since $F$ is symmetric, both expressions are equal and thus $(X,B)$ is commutative.
\end{proof}
\begin{remark}
 It follows from the proof of Proposition \ref{prop:propinj1} that if $(X,F)$ is the $n$-ary extension of a semilattice, then this semilattice is $(X,B)$.
\end{remark}
In the same spirit, the map $\l$ also enables us to characterize another class of symmetric $n$-ary bands. We recall that a group $(X,\ast)$ with neutral element $e$ has \emph{bounded exponent} if there exists an integer $m\geq 1$ such that the $m$-fold product $x\ast\cdots\ast x$ is equal to $e$ for any $x\in X$. In that case, the exponent of the group is the smallest integer having this property.
\begin{proposition}\label{prop:group1}
Let $(X,F)$ be a symmetric $n$-ary band. The following conditions are equivalent.
\begin{enumerate}
\item[(i)] The map $\l$ is constant (i.e., $\l_x$ is the identity map of $X$ for any $x\in X$).
 \item[(ii)] The band $(X,F)$ is the $n$-ary extension of an Abelian group $(X,\ast)$ (and in particular the exponent of $(X,\ast)$ divides $n-1$).
 \item[(iii)] The band $(X,B)$ is a right zero semigroup.
\end{enumerate}
\end{proposition}
\begin{proof}
Let us first analyze conditions (i) and (ii). If $\l$ is constant, then we have $\l_x(y)=\l_y(y)=y$ for every $x,y\in X$, and thus $\l_x$ is the identity map of $X$ for every $x$. If $(X,F)$ is the $n$-ary extension of an Abelian group $(X,\ast)$, then
we have $x=F(n\Cdot x)=x\ast\cdots\ast x$ so the exponent of $(X,\ast)$ must divide $n-1$.

The equivalence of conditions (i) and (iii) follows directly from the definitions.

Let us show that (i) implies (ii). Using (i) and the symmetry of $F$, we obtain that every element $x\in X$ is neutral for $F$ (see Section \ref{Sec2.1}). It is not difficult to prove that the reduction of $F$ with respect to any neutral element defines an Abelian group whose exponent divides $n-1$. This was done in detail in \cite[Theorem 1.3]{CouDevMarMat19} for $n\geqslant 3$ and is obvious for $n=2$.
Finally, we show that (ii) implies (i). If $F$ is the $n$-ary extension of $(X,\ast)$, then 
\[\l_x(y)=F((n-1)\Cdot x,y)= x\ast\cdots\ast x\ast y=y,\]
so $\l_x$ is the identity map, and (i) holds true. 
\end{proof}
\section{Semilattice decomposition and induced group structures}\label{Sec4}
In this section we still consider a symmetric $n$-ary band $(X,F)$ and we investigate further relations between this structure and its associated binary band $(X,B)$.
On the one hand, when the map $\l$ associated with $B$ is not injective, it is natural to consider a quotient, and identify the elements of $X$ that have the same image by $\l$. On the other hand, considering a right normal band $(X,B)$, it is natural to study its strong semilattice decomposition (see Proposition \ref{prop:rightn}) induced by its smallest semilattice congruence (defined by \eqref{eq:sigma}). We will now show that both constructions lead to the same congruence for $(X,B)$. Moreover, we will prove that it is also a congruence for $F$ and that the restrictions of $F$ to each equivalence class is reducible to an Abelian group whose exponent divides $n-1$. 

From now on, we will denote by $\sigma$ the least semilattice congruence for the band $(X,B)$ associated with $(X,F)$. 
For every $x\in X$, we simply denote by $[x]$ its equivalence class modulo $\sigma$. 

The following result shows how to express $\sigma$ using the map $\l$ and properties of $F$.
\begin{proposition}\label{prop:sigma1}
For every $x,y\in X$, the following conditions are equivalent.
 \begin{enumerate}
  \item[(i)] $x\sigma y$.
  \item[(ii)] $\l_x=\l_y$.
  \item[(iii)] There exist $t,t'\in X$ such that $y=\l_x(t)$ and $x=\l_y(t')$.
  \item[(iv)] We have $y=F(x,x_2,\ldots,x_n)$ and $x=F(y,y_2,\ldots,y_n)$ for some $x_i,y_i\in X$ ($i\in\{2,\ldots,n\})$. 
 \end{enumerate}
\end{proposition}
\begin{proof}
Let us show that (i) implies (ii). By \eqref{eq:sigma}, (i) is equivalent to the conditions $\l_x(y)=y$ and $\l_y(x)=x$. By Proposition \ref{prop:impo} we have
\[\l_x(t)=\l_{\l_y(x)}(t)=\l_{\l_x(y)}(t)=\l_y(t),\qquad t\in X.\]
To show that (ii) implies (iii), we observe that $y=\l_y(y)$, so that (ii) implies $y=\l_x(y)$, and in the same way we obtain $x=\l_y(x)$. It follows directly from the definition of $\l$ that (iii) implies (iv). Finally, we use the associativity and idempotency of $F$ to show that (iv) implies the relations $\l_x(y)=y$ and $\l_y(x)=x$. These conditions are equivalent to (i) by \eqref{eq:sigma}.
\end{proof}
The congruence $\sigma$ was built using the binary band $(X,B)$. We will now show that it also defines a decomposition of $(X,F)$.
\begin{proposition}\label{prop:congruence}
The equivalence relation $\sigma$ is a congruence for $F$.
\end{proposition}
\begin{proof}
Assume that $x_i\sigma y_i$ for every $i\in\{1,\ldots,n\}$. By Propositions \ref{prop:impo3} and \ref{prop:sigma1}, we have
$$
\l_{F(x_1,\ldots,x_n)} = \l_{x_1}\circ\cdots\circ\l_{x_{n}} = \l_{y_1}\circ \cdots \circ \l_{y_n} = \l_{F(y_1,\ldots,y_{n})}.
$$
By Proposition \ref{prop:sigma1}, we have $F(x_1,\ldots,x_n)\sigma F(y_1,\ldots,y_{n})$.
\end{proof}
Since $\sigma$ is a congruence for $F$ and for $B$, we can define the induced operations $F^\sigma$ and $B^\sigma$ on $X/\sigma$ in the usual way by setting for all $x_1,\ldots,x_n\in X$,
\begin{equation}\label{FsigmaBsigma}
F^\sigma([x_1],\ldots,[x_n])=[F(x_1,\ldots,x_n)]\quad\mbox{and}\quad B^\sigma([x_1],[x_2])=[B(x_1,x_2)].
\end{equation}
\begin{example}\label{ex:example3}
For both structures presented in Examples \ref{ex:1.1} and \ref{ex:2.6}, we only have $\l_3=\l_4$ so $[1]=\{1\}$, $[2]=\{2\}$, and $[3]=\{3,4\}$. The semilattice $(X/\sigma,B^\sigma)$ is obtained by identifying $3$ and $4$ in the tables given in Example \ref{ex:2.6}. We see that $(X/\sigma,B^\sigma)$ is a semilattice whose Hasse diagram is given in Figure \ref{fig:nonass}.
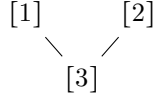
\begin{figure}[h]
\begin{center}
\begin{tikzpicture}[scale=0.5]
\clip (-1.9,-2.1) rectangle (1.9,0.5);
\node (3) at (0,-1.71) {$[3]$};
\node (1) at (-1.5,0) {$[1]$};
\node (2) at (1.5,0) {$[2]$};
\draw (1)--(3);
\draw (2) --(3);
\end{tikzpicture}
\end{center}
\caption{Hasse diagram of $(X/\sigma,B^\sigma)$\label{fig:nonass}}
\end{figure}
\end{example}

We will show in the next proposition that $(X/\sigma,F^\sigma)$ is also a symmetric $n$-ary band. We will then denote by $(X/\sigma,B_{F^\sigma})$ its associated band. We now give the most important properties of these maps.
\begin{proposition}\label{prop:Fsymb}
The pair $(X/\sigma, F^\sigma)$ is a symmetric $n$-ary band. Moreover, we have $B_{F^\sigma}=B^\sigma$. In particular $(X/\sigma, F^\sigma)$ is the $n$-ary extension of the semilattice $(X/\sigma, B^\sigma)$.
\end{proposition}
\begin{proof}
The proof of the first statement is straightforward, using the properties of $F$ and the definition of $F^\sigma$. Moreover, for every $x,y\in X$ we compute
\[B_{F^\sigma}([x],[y])=F^\sigma((n-1)\Cdot [x],[y])=[F((n-1)\Cdot x,y)]=[B(x,y)]=B^\sigma([x],[y]),\]
which shows the second statement. Also, since $\sigma$ is a semilattice congruence for $B$, $B^\sigma$ is symmetric. Thus $B_{F^\sigma}$ is symmetric and the result then follows from Proposition \ref{prop:propinj1}.
\end{proof}
This proposition suggests to generalize the definition of semilattice congruences for $n$-ary 
semigroups $(X,F)$.
\begin{definition}
An $n$-ary semigroup $(X,F)$ is an $n$-ary semilattice if it is the $n$-ary extension of a semilattice. A congruence $\sim$ on an $n$-ary semigroup $(X,F)$ is an $n$-ary semilattice congruence if $(X/\sim,\tilde{F})$ is an $n$-ary semilattice.
\end{definition}
Proposition \ref{prop:Fsymb} shows that $\sigma$ is an $n$-ary semilattice congruence on $(X,F)$. 

Since $\sigma$ is a congruence for $F$, this operation restricts to each equivalence class. We now analyze the properties of this restriction. 

\begin{proposition}\label{prop:group}
For any $x\in X$, $([x],F|_{[x]^n})$ is the $n$-ary extension of an Abelian group whose exponent divides $n-1$.
\end{proposition}
\begin{proof}
Since the restriction of operations preserves associativity, symmetry, and idempotency, for every $x\in X$, $([x],F|_{[x]^n})$ is a symmetric $n$-ary band. Its associated binary band is defined by $B_{F|_{[x]^n}}(y,z)=F|_{[x]^n}((n-1)\Cdot y,z)=B_F(y,z)$, for every $y,z\in[x]$. It is thus the restriction of $B_F$ to $[x]^2$. By Proposition \ref{prop:rightn} this restriction defines a right zero semigroup. The result then follows from Proposition \ref{prop:group1}.
\end{proof}
\begin{example}\label{ex:example4}
For both structures presented in Example \ref{ex:1.1}, the restriction to $[3]$ is isomorphic to the ternary extension of $(\mathbb{Z}_2,+)$. Observe that selecting $3$ or $4$ as the neutral element of this group leads to the same ternary extension.
\end{example}
Recall that the strong semilattice decomposition of $(X,B)$ defines a family of right zero semigroups homomorphisms ($\varphi_{[x],[y]},[x]\geqslant [y])$ defined by $\varphi_{[x],[y]}=\l_y|_{[x]}$. In the next proposition we study the compatibility of these maps with respect to the structure induced by $F$ on the classes $[x]$ and $[y]$.
\begin{proposition}\label{prop:homgroupphi}
For every $x,y\in X$ such that $[x]\geqslant [y]$, the map $\varphi_{[x],[y]}$ is a homomorphism from $([x],F|_{[x]^n})$ to $([y],F|_{[y]^n})$.
\end{proposition}
\begin{proof}
This follows from Proposition \ref{prop:end} and from the definition of the map $\varphi_{[x],[y]}$.
\end{proof}
As in the case of binary structures, the existence of an $n$-ary semilattice congruence induces a decomposition of $X$ as a disjoint union of the family of subsets $\{X_{\alpha}\colon \alpha\in Y\}$ where $(Y,\curlywedge)$ is a semilattice. Moreover, the operation $F$ induces on each $X_\alpha$ an $n$-ary semigroup operation $F_\alpha$ and we have 
\begin{equation}\label{eq:semisem2}
F(X_{\alpha_1}\times \cdots \times X_{\alpha_n}) \subseteq X_{\alpha_1 \curlywedge \cdots \curlywedge \alpha_n}, \qquad \alpha_1,\ldots,\alpha_n\in Y.
\end{equation}
We call such a structure an \emph{$n$-ary semilattice $(Y,\curlywedge^{n-1})$ of $n$-ary semigroups $(X_{\alpha},F_{\alpha})$} and we write $(X,F) = ((Y,\curlywedge^{n-1});(X_{\alpha},F_{\alpha}))$. We also say that $(X,F)$ is an $n$-ary semilattice of $n$-ary semigroups.

The fact that an $n$-ary groupoid is an $n$-ary semilattice of $n$-ary semigroups is not sufficient to ensure that it is an $n$-ary semigroup. We need to introduce a generalization of the strong semilattice decomposition. This is done in the following definition.
\begin{definition}\label{def:strongnsem}
Let $(X,F) = ((Y,\curlywedge^{n-1});(X_{\alpha},F_{\alpha}))$ be an $n$-ary semilattice of $n$-ary semigroups. Suppose that for any $\alpha,\beta\in Y$ such that $\alpha \geqslant \beta$ there is a homomorphism $\varphi_{\alpha,\beta}\colon X_{\alpha} \to X_{\beta}$ such that the following conditions hold.
\begin{enumerate}
\item[(a)] The map $\varphi_{\alpha,\alpha}$ is the identity on $X_{\alpha}$.
\item[(b)] For any $\alpha,\beta,\gamma\in Y$ such that $\alpha \geqslant \beta \geqslant \gamma$ we have $\varphi_{\beta,\gamma}\circ \varphi_{\alpha,\beta} = \varphi_{\alpha,\gamma}$.
\item[(c)] For any $(x_1,\ldots,x_n)\in X_{\alpha_1}\times \cdots \times X_{\alpha_n}$ we have
$$
F(x_1,\ldots,x_n) = F_{\alpha_1\curlywedge\cdots \curlywedge \alpha_n}(\varphi_{\alpha_1,\alpha_1\curlywedge\cdots \curlywedge \alpha_n}(x_1),\ldots,\varphi_{\alpha_n,\alpha_1\curlywedge\cdots \curlywedge \alpha_n}(x_n)).
$$
\end{enumerate}
Then $(X,F)$ is said to be a \emph{strong $n$-ary semilattice $(Y,\curlywedge^{n-1})$ of $n$-ary semigroups $(X_{\alpha},F_{\alpha})$}. In this case we write $(X,F) = ((Y,\curlywedge^{n-1});(X_{\alpha},F_{\alpha});\varphi_{\alpha,\beta})$ (or simply $[Y,(X_\alpha,F_\alpha),\varphi_{\alpha,\beta}])$ and we also say that $(X,F)$ is a strong $n$-ary semilattice of $n$-ary semigroups.
\end{definition}
We can now state our main structure theorem for symmetric $n$-ary bands.
\begin{theorem}\label{thm:mainnband}
An $n$-ary groupoid $(X,F)$ is a symmetric $n$-ary band if and only if is a strong $n$-ary semilattice of $n$-ary extensions of Abelian groups whose exponents divide $n-1$.
\end{theorem}
In order to prove Theorem \ref{thm:mainnband} we will make use of the following proposition that holds in a more general context.
\begin{proposition}\label{prop:carstrongnsem}
If $(X,F) = ((Y,\curlywedge^{n-1});(X_{\alpha},F_{\alpha});\varphi_{\alpha,\beta})$ is a strong $n$-ary semilattice of $n$-ary semigroups, then it is an $n$-ary semigroup.
\end{proposition}
\begin{proof}
Let $x_1,\ldots,x_{2n-1}\in X$, $i\in\{1,\ldots,n\}$ and let us consider the expression
\[F(x_1,\ldots,x_{i-1},F(x_i,\ldots,x_{i+n-1}),x_{i+n},\ldots,x_{2n-1}).\]
For $1\leq k\leq 2n-1$ there exists $\alpha_k\in Y$ such that $x_k$ belongs to $X_{\alpha_k}$. Then we set $\alpha'=\alpha_{i}\curlywedge\cdots\curlywedge\alpha_{i+n-1}$ and we have
\[F(x_i,\ldots,x_{i+n-1})=F_{\alpha'}(\varphi_{\alpha_i,\alpha'}(x_i),\ldots,\varphi_{\alpha_{i+n-1},\alpha'}(x_{i+n-1}))\in X_{\alpha'}.\]
We denote this element by $x$ and we set $\alpha=\alpha_1\cw\cdots\cw\alpha_{2n-1}$. We then have
\begin{multline*}
F(x_1\ldots,x_{i-1},x,x_{i+n},\ldots,x_{2n-1})\\
= F_\alpha(\varphi_{\alpha_1,\alpha}(x_1),\ldots,\varphi_{\alpha_{i-1},\alpha}(x_{i-1}),\varphi_{\alpha',\alpha}(x),\varphi_{\alpha_{i+n},\alpha}(x_{i+n}),\ldots,\varphi_{\alpha_{2n-1},\alpha}(x_{2n-1})).
\end{multline*}
Also, using the fact that $\varphi_{\alpha',\alpha}$ is a homomorphism and condition (b) of Definition \ref{def:strongnsem}, we have that
\[\varphi_{\alpha',\alpha}(x)=F_\alpha(\varphi_{\alpha_i,\alpha}(x_i),\ldots,\varphi_{\alpha_{i+n-1},\alpha}(x_{i+n-1})).\]
Thus, the associativity of $F$ follows from the associativity of $F_\alpha$ and from the fact that $\alpha$ is independent of $i$.
\end{proof}
\begin{proof}[Proof of Theorem \ref{thm:mainnband}]
 We already proved all the necessary properties to ensure that every symmetric $n$-ary band $(X,F)$ is a strong $n$-ary semilattice of Abelian groups whose exponents divide $n-1$:  
 We first associated with it a right normal band $(X,B)$ (see Definition \ref{def:binary} and
 Proposition \ref{prop:impo}). The smallest semilattice congruence $\sigma$ on this band (defined by \eqref{eq:sigma}) is an $n$-ary semilattice congruence on $(X,F)$ by Proposition \ref{prop:Fsymb}. Proposition \ref{prop:group} shows that the restriction of $F$ to each class of the decomposition defines an $n$-ary extension of an Abelian group whose exponent divides $n-1$. Proposition \ref{prop:rightn} implies the existence of maps $\varphi_{[x],[y]}$ from $[x]$ to $[y]$ when $[x]\geqslant[y]$. These maps are defined by $\varphi_{[x],[y]}=\l_y|_{[x]}$ and they satisfy conditions (a) and (b) of Definition \ref{def:strongnsem}. Proposition \ref{prop:homgroupphi} then shows that these maps are homomorphisms.
 
It remains to prove condition (c) of Definition \ref{def:strongnsem}. We use Corollary \ref{cor:proj}: if $x_i\in X_{\alpha_i}$, we have 
\[F(x_1,\ldots,x_n)=F(\l_{F(x_1,\ldots,x_n)}(x_1),\ldots,\l_{F(x_1,\ldots,x_n)}(x_n)).\] The definition of $F^\sigma$ implies $[F(x_1,\ldots,x_n)]=F^\sigma([x_1],\ldots,[x_n])$. Moreover, by Proposition \ref{prop:Fsymb}, $F^\sigma$ is the $n$-ary extension of $\cw=B^\sigma$, the semilattice operation on $X/\sigma$, so we have 
\[[F(x_1,\ldots,x_n)]=[x_1]\cw\cdots\cw[x_n]=\alpha_1\cw\cdots\cw\alpha_n\leqslant \alpha_i\quad i\in\{1,\ldots,n\},\]
and so setting $\alpha=\alpha_1\cw\cdots\cw\alpha_n$ we have $\l_{F(x_1,\ldots,x_n)}(x_i)=\varphi_{\alpha_i,\alpha}(x_i)\in X_\alpha$ for every $i\in\{1,\ldots,n\}$.

Now, let us show the converse statement. Assume that $(X,F) = [Y,(X_{\alpha},F_{\alpha}),\varphi_{\alpha,\beta}]$ is a strong $n$-ary semilattice of $n$-ary extensions of Abelian groups whose exponents divide $n-1$. The associativity of $F$ follows from Proposition \ref{prop:carstrongnsem}. The idempotency of $F$ follows from conditions (a) and (c) of Definition \ref{def:strongnsem} and from the idempotency of the $n$-ary operations $F_\alpha$. Finally, the symmetry of $F$ follows from condition (c) of Definition \ref{def:strongnsem} and the symmetry of the $n$-ary operations $F_\alpha$ and $\cw^{n-1}$.
\end{proof}
In view of this result, in order to build symmetric $n$-ary bands, we have to consider Abelian groups whose exponents divide $n-1$, and build homomorphisms between the $n$-ary extensions of such groups. These homomorphisms are described in the next result.
\begin{proposition}\label{prop:homomorphims}
Let $(X_1,\ast_1)$ and $(X_2,\ast_2)$ be two Abelian groups whose exponents divide $n-1$ and denote by $F_1$ and $F_2$ the $n$-ary extensions of $\ast_1$ and $\ast_2$ respectively. For every group homomorphism $\psi\colon X_1\to X_2$ and every $g_2\in X_2$, the map $h\colon X_1\to X_2$ defined by
\[h(x)=g_2\ast_2\psi(x),\quad x\in X_1\]
is a homomorphism of $n$-ary semigroups.

Conversely, every homomorphism from $(X_1,F_1)$ to $(X_2,F_2)$ is obtained in this way.
\end{proposition}
\begin{proof}
For all $x_1,\ldots,x_n\in X_1$ we compute 
\[h(F_1(x_1,\ldots,x_n))=g_2\ast_2\psi(F_1(x_1,\ldots,x_n))=g_2\ast_2\psi(x_1\ast_1\cdots\ast_1x_n).\]
Since $\psi$ is a group homomorphism, this expression is equal to $g_2\ast_2\psi(x_1)\ast_2\cdots\ast_2\psi(x_n)$.
Moreover, using the definition of $h$ and the commutativity of $\ast_2$ 
\[F_2(h(x_1),\ldots,h(x_n))=h(x_1)\ast_2\cdots\ast_2h(x_n)=g_2\ast\cdots\ast g_2\ast_2\psi(x_1)\ast_2\cdots\ast_2\psi(x_n).\]
The first part of the result follows since the exponent of $(X_2,\ast_2)$ divides $n-1$.

For the second part, we consider a homomorphism $h$ of $n$-ary semigroups from $(X_1,F_1)$ to $(X_2,F_2)$, and we denote by $i_1$ the inverse of $h(e_1)$ with respect to $\ast_2$. Then the map $\psi\colon X_1\to X_2$ defined by 
\[\psi(x)=i_1\ast_2 h(x),\quad x\in X_1\]
is a group homomorphism from $(X_1,\ast_1)$ to $(X_2,\ast_2)$. Indeed, for $x,y\in X_1$ we have
\[\psi(x\ast_1 y)=i_1\ast_2h(x\ast_1 y)=i_1\ast_2h(x\ast_1 y\ast_1e_1\ast_1\cdots\ast_1 e_1)\]
where $e_1$ is repeated $n-2$ times. This expression is equal to $i_1\ast_2h(F_1(x,y,(n-2)\Cdot e_1))$.
Since $h$ is a homomorphism, it is also equal to
\[i_1\ast_2F_2(h(x),h(y),(n-2)\Cdot h(e_1)))=i_1\ast_2x\ast_2y\ast_2 h(e_1)\ast_2\cdots \ast_2h(e_1),\]
where $h(e_1)$ is repeated $n-2$ times. Since the exponent of $(X_2,\ast_2)$ divides $n-1$ it follows that $\psi$ is a group homomorphism.
\end{proof}

\section{Reducibility of symmetric $n$-ary bands}\label{Sec5}
In this section, we use the structure theorem that we developed in the previous section in order to analyze the reducibility problem for symmetric $n$-ary bands. 
 We thus consider a symmetric $n$-ary band $(X,F)=[Y,(X_\alpha,F_\alpha),\varphi_{\alpha,\beta}]$, with associated binary band $(X,B)$. If there is no risk of confusion, we will denote a binary operation by $G\colon X^2\to X$ in a multiplicative way, setting $G(x,y)=x\ast y$. If we use this notation, we then denote by $x^k$ the $k$-fold product $x\ast\cdots\ast x$.

\begin{proposition}\label{reduc}
 If $F$ is reducible to an associative operation $G\colon X^2 \to X$, then the following assertions hold.
 \begin{enumerate}
  \item[(i)] $G$ is surjective and symmetric;
  \item[(ii)] The map $\l_x$ is an endomorphism of $(X,G)$, for every $x\in X$;
  \item[(iii)] We have $\l_{G(x,y)}=\l_x\circ\l_y=\l_{B(x,y)}$ for every $x,y\in X$;
  \item[(iv)] The congruence $\sigma$ associated with $F$ is a congruence for $G$. Moreover, the associated operation $G^\sigma$ on $X/\sigma$ is equal to the quotient operation $B^\sigma$.
 \end{enumerate}
\end{proposition}

\begin{proof}
The surjectivity of $G$ follows from the idempotency of $F$. It was already observed in \cite[Fact 4.1]{CouDev}. The symmetry of $G$ follows from \cite[Lemma 3.6]{DevKisMar19}. 
In order to prove (ii) we compute, for $x,y,z\in X$:
\[\l_x(y\ast z)=F((n-1)\Cdot x,y\ast z)=x^{n-1}\ast (y\ast z)=(x^{n-1}\ast y)\ast z=\l_x(y)\ast z.\]
The result then follows directly from the symmetry of $G$ and Proposition \ref{prop:end}.

Property (iii) is proved in the same way: for every $x,y,z\in X$ we have 
\[\l_{x\ast y}(z)=F((n-1)\Cdot (x\ast y),z)=(x\ast y)^{n-1}\ast z.\]
Using the associativity and the symmetry of $G$, we can rewrite this expression as \[x^{n-1}\ast y^{n-1}\ast z=\l_x(\l_y(z)).\] The last equality in (iii) follows from Proposition \ref{prop:impo} and from the definition of $B$.

Finally, from (iii) and Proposition \ref{prop:sigma1}, we obtain $G(x,y)\sigma B(x,y)$ for every $x,y\in X$. Since $\sigma$ is a congruence for $B$, it is also a congruence for $G$. The quotient operation can be easily computed:
\[G^\sigma([x],[y])=[G(x,y)]=[B(x,y)]=B^\sigma([x],[y]),\quad x,y\in X,\]
and the proof is complete.
\end{proof}
It follows from Proposition \ref{reduc} that if $F$ is reducible to $G$, then $G$ induces an operation $G|_{[x]^2}$ on every $\sigma$-class $[x]$ of $F$. This operation is a reduction of $F|_{[x]^n}$. It is therefore natural to study the properties of the reduction of such operations. This is done in the following result.
\begin{proposition}\label{prop:reduc2}
If $(X,F)$ is the $n$-ary extension of an Abelian group $(X,G_1)$ whose exponent divides $n-1$, then every reduction $(X,G_2)$ of $F$ is a group that is isomorphic to $(X,G_1)$. All the reductions of $(X,F)$ are obtained by using \eqref{eq:dud} with any element $e$ of $X$.
\end{proposition}
\begin{proof}
This is a direct consequence of \cite[Proposition 1.4]{CouDevMarMat19}.
\end{proof}
We are now able to analyze the reducibility of symmetric $n$-ary bands.
\begin{proposition}\label{prop:reduc}
A symmetric $n$-ary band $(X,F)=[Y,(X_\alpha,F_\alpha),\varphi_{\alpha,\beta}]$ is reducible to a semigroup if and only if there exists a map $e\colon Y\to X$ such that 
\begin{enumerate}
 \item[(i)] For every $\alpha\in Y$, $e(\alpha)=e_\alpha$ belongs to $X_\alpha$;
 \item[(ii)] For every $\alpha,\beta\in Y$ such that $\alpha\geqslant\beta$, we have $\varphi_{\alpha,\beta}(e_\alpha)=e_\beta$.
\end{enumerate}
Moreover, when $(X,F)$ is reducible to a semigroup, a reduction is given by the semigroup decomposed as $[Y,(X_\alpha,G_\alpha),\varphi_{\alpha,\beta}]$, where $G_\alpha$ is the reduction of $F_\alpha$ with respect to $e_\alpha$.
\end{proposition}
\begin{proof}
 Assume first that $(X,F)$ is reducible to a semigroup $(X,G)$. By Proposition \ref{reduc}, we have a decomposition of $(X,G)$ as $[Y,(X_\alpha,G_\alpha),\varphi_{\alpha,\beta}]$, where $Y=X/\sigma$, the sets $X_\alpha$ are the $\sigma$-classes of $X$, and the maps $\varphi_{\alpha,\beta}$ are given, for $\alpha=[x]\geqslant \beta=[y]$, by $\varphi_{\alpha,\beta}=\l_y|_{[x]}$. Using (ii) of Proposition \ref{reduc}, a direct computation show that these maps are group homomorphisms. Moreover, by Proposition \ref{prop:reduc2}, the restriction $G|_{[x]^2}$ of $G$ to each $\sigma$-class $[x]$ is a reduction of the restriction of $F$ to this class, and is therefore associated with an element $e_{[x]}$ in this class. Since $e_{[x]}$ is the unit of the group $([x],G|_{[x]^2})$ and $\varphi_{[x],[y]}$ is a group homomorphism, we have $\varphi_{[x],[y]}(e_{[x]})=e_{[y]}$, and conditions (i) and (ii) are satisfied for $e([x])=e_{[x]}$.
 
 Conversely, assume that (i) and (ii) are satisfied. For every $\alpha\in Y$, denote by $G_\alpha$ the reduction of $F_\alpha$ associated with $e_\alpha$. If $\alpha\geqslant \beta$, $\varphi_{\alpha,\beta}$ is a homomorphism from $(X_\alpha,F_\alpha)$ to $(X_\beta,F_\beta)$. By Proposition \ref{prop:homomorphims} and condition (ii), $\varphi_{\alpha,\beta}$ is a group homomorphism. Conditions (a) and (b) of Definition \ref{def:strong} are then satisfied, and it follows that $[Y,(X_\alpha,G_\alpha),\varphi_{\alpha,\beta}]$ defines a semigroup $G$ by condition (c) of the same definition. It now remains to check that $(X,G)$ is a reduction of $(X,F)$. To keep the notation light, we set $G_\alpha(x,y)=x\ast_\alpha y$ and $G(x,y)=x\ast y$. We first observe that, when $\alpha,\alpha_1,\alpha_2\in Y$ are such that $\alpha\leqslant \alpha_1\cw \alpha_2$, we have 
 \begin{equation}\label{eq:phialpha}
  \varphi_{\alpha_1,\alpha}(x_1)\ast_\alpha \varphi_{\alpha_2,\alpha}(x_2)=\varphi_{\alpha_1\cw\alpha_2,\alpha}(x_1\ast x_2).
 \end{equation}
This relation is obtained by decomposing  $\varphi_{\alpha_i,\alpha}(x_i)$ as  $\varphi_{\alpha_1\cw\alpha_2,\alpha}(\varphi_{\alpha_i,\alpha_1\cw\alpha_2}(x_i))$ in the left-hand side, for $i\in\{1,2\}$, using that $\varphi_{\alpha_1\cw\alpha_2,\alpha}$ is a homomorphism and finally using the definition of $\ast$. Then, if $x_i\in X_{\alpha_i}$ for every $i\in\{1,\ldots,n\}$, setting $\alpha=\alpha_1\cw\cdots\cw\alpha_n$ we have by definition
\[F(x_1,\ldots,x_n)=F_\alpha(\varphi_{\alpha_1,\alpha}(x_1),\ldots,\varphi_{\alpha_n,\alpha}(x_n)).\]
Since $F_\alpha$ is the $n$-ary extension of $\ast_\alpha$, we also have 
\[F(x_1,\ldots,x_n)=\varphi_{\alpha_1,\alpha}(x_1)\ast_\alpha\ldots\ast_\alpha\varphi_{\alpha_n,\alpha}(x_n).\]
Then using \eqref{eq:phialpha} we prove by induction 
\[F(x_1,\ldots,x_n)=\varphi_{\alpha_1\cw\cdots\cw\alpha_i}(x_1\ast\cdots\ast x_i)\ast_\alpha\varphi_{\alpha_{i+1},\alpha}(x_{i+1})\ast_\alpha\ldots\ast_\alpha\varphi_{\alpha_n,\alpha}(x_n),\]
for every $i\in\{1,\ldots,n\}$. Considering $i=n$ leads to the desired result.
\end{proof}
\begin{remark}
In the proof of Proposition \ref{prop:reduc} we could also show that when conditions (i) and (ii) are satisfied, we can adjoin a neutral element to $(X,F)$, and obtain a reduction of $(X,F)$ to a semigroup using Theorem 1 of \cite{DudMuk06}. Our approach enables us to describe the strong semilattice decomposition of the reduction explicitly.
\end{remark}
\begin{example}
For the structures of Example \ref{ex:1.1}, the only non obvious homomorphisms are give by $\varphi_{[1],[3]}=\l_3|_{[1]}$ and $\varphi_{[2],[3]}=\l_3|_{[2]}$. 
\begin{enumerate}                                                                                                                                                \item For $(X,F_1)$, we have $\varphi_{[1],[3]}(1)=\l_3(1)=4$ and $\varphi_{[2],[3]}(2)=\l_3(2)=3$;
\item For $(X,F_2)$, we have $\varphi_{[1],[3]}(1)=\l_3(1)=4$ and $\varphi_{[2],[3]}(2)=\l_3(2)=4$.
\end{enumerate}
By Proposition \ref{prop:homomorphims} these are homomorphisms from the ternary extension of the trivial group to the ternary extension of $(\Z_2,+)$. It is easy to see that $(X,F_1)$ and $(X,F_2)$ are the strong ternary semilattices associated with the semilattice given in Example \ref{ex:example3} and the ternary extensions of groups and homomorphisms given here. It follows from Theorem \ref{thm:mainnband} that $(X,F_1)$ and $(X,F_2)$ are symmetric ternary bands.
Finally, we can use Proposition \ref{prop:reduc} to analyze the reducibility problem for $(X,F_1)$ and $(X,F_2)$.
\begin{enumerate}
 \item For $(X,F_1)$ we must have $e([1])=1$ and $e([2])=2$. Then we must have $e([3])=\varphi_{[1],[3]}(1)=4$ but also $e([3])=\varphi_{[2],[3]}(2)=3$, a contradiction. So $(X,F_1)$ is not reducible to a semigroup.
 \item For $(X,F_2)$ the map $e$ defined by $e([1])=1$, $e([2])=2$ and $e([3])=4$ satisfies the conditions of Proposition \ref{prop:reduc} and so $(X,F_2)$ is reducible to a semigroup.
\end{enumerate}
\end{example} 

\section*{Acknowledgments}

The authors would like to thank M. Couceiro, J.-L. Marichal and N. Z\'ena\"idi for fruitful discussions. The first author is supported by the Luxembourg National Research Fund under the project PRIDE 15/10949314/GSM.

\end{document}